\documentclass[12pt]{amsart}
\usepackage{amssymb,amsfonts,amsmath,amsthm,cite,verbatim}
\usepackage{xcolor}
\usepackage{tikz}
\usepackage{graphicx}
\usepackage[left=2.5cm, right=2.5cm, top=3.5cm, bottom=3.5cm]{geometry}

\usepackage{hyperref}

\usepackage[normalem]{ulem}
\newtheorem{theorem}{Theorem}[section]
\newtheorem{thm}[theorem]{Theorem}
\newtheorem{prop}[theorem]{Proposition}
\newtheorem{lem}[theorem]{Lemma}

\newtheorem{cor}[theorem]{Corollary}

\newtheorem{conj}[theorem]{Conjecture}

\makeatletter \@addtoreset{equation}{section}

\newcommand{\qbinom}[2]{\genfrac{[}{]}{0pt}{}{#1}{#2}}
\newcommand{\lrq}[3]{\left(\frac{#1}{#2}#3\right)}

\newcommand{\Mid}{\:|\:}  
\DeclareMathOperator*{\CT}{CT}

\newcommand{\CC}{\mathbb{C}}

\begin{document}

\title{On the $q$-Dyson orthogonality problem}

\author{Yue Zhou}

\address{School of Mathematics and Statistics, Central South University,
Changsha 410075, P.R. China}

\email{zhouyue@csu.edu.cn}



\begin{abstract}
By combining the Gessel--Xin method with plethystic substitutions,
we obtain a recursion for a symmetric function generalization of the
$q$-Dyson constant term identity also known as
the Zeilberger--Bressoud $q$-Dyson theorem.
This yields a constant term identity which generalizes the non-zero part
of Kadell's orthogonality ex-conjecture and a result
of K\'{a}rolyi, Lascoux and Warnaar.

\end{abstract}

\maketitle

\section{Introduction}

The study of constant term identities can be traced back to a 1962 paper on random matrices and the theory of statistical levels of complex systems by Freeman Dyson~\cite{dyson}.
In the course of this work he conjectured that for non-negative integers $a_0,\dots,a_n$,
\begin{equation}\label{thm-Dyson}
\CT_{x} \,\prod_{0\leq i\neq j\leq n}(1-x_i/x_j)^{a_i}=\frac{(a_0+\cdots+a_n)!}{a_0!\cdots a_n!},
\end{equation}
where $\displaystyle\CT_x$ denotes taking the constant term with respect to $x:=(x_0,\dots,x_n)$.
Dyson's conjecture was soon proved by Gunson \cite{gunson} and Wilson \cite{wilson}.
Subsequently, an elegant proof using Lagrange interpolation was given by Good \cite{good},
and much later, Zeilberger gave a combinatorial
proof using tournaments \cite{zeil}.
These days Dyson's ex-conjecture is usually referred as the Dyson constant term identity.

In this introduction,
we first briefly review the history of the Dyson constant term identity and some of its generalizations. Then we state our main result, a symmetric function generalization of the Dyson constant term identity, related to Kadell's orthogonality (ex-)conjecture.
We conclude the introduction by outlining the main
new ideas used in this paper.

In 1975 Andrews \cite{andrews1975} conjectured the following $q$-analogue of \eqref{thm-Dyson}:
\begin{equation}\label{q-Dyson}
\CT_x \,
\prod_{0\leq i<j\leq n}
(x_i/x_j;q)_{a_i}(qx_j/x_i;q)_{a_j}=
\frac{(q;q)_{a_0+\cdots+a_n}}{(q;q)_{a_0}(q;q)_{a_1}\cdots(q;q)_{a_n}},
\end{equation}
where
$(z;q)_k:=(1-z)(1-zq)\dots(1-zq^{k-1})$ is a
$q$-shifted factorial.
Andrews' $q$-Dyson conjecture was first proved in 1985 by Zeilberger and Bressoud \cite{zeil-bres1985}, who generalized Zeilberger's method of tournaments mentioned above.
Twenty years later Gessel and Xin \cite{gess-xin2006} gave a second proof using formal Laurent series, and
then, in 2014, K\'{a}rolyi and Nagy \cite{KN} discovered a very short and elegant proof using multivariable Lagrange interpolation.
Finally, an inductive proof was found by Cai \cite{cai} by adding additional parameters to the problem.

In 1982, Macdonald realised that the equal parameter case of the $q$-Dyson identity, i.e., $a_0=a_1=\cdots=a_n=k$, can be formulated as a combinatorial
identity for the root system $\mathrm{A}_n$.
This led him to conjecture a constant term identity for arbitrary root systems \cite{macdonald82}:
\[
\CT \,\prod_{\alpha\in R^+}(e^{-\alpha};q)_k(qe^{\alpha};q)_k=
\prod_{i=1}^r\qbinom{d_ik}{k}.
\]
Here $R$ is a reduced irreducible finite root system of rank $r$, $R^+$ is the set of positive roots,
$d_1,\dots,d_r$ are the degrees of the fundamental invariants,
and $\qbinom{n}{k}$ is a $q$-binomial coefficient. 
Initially, many cases of Macdonald's conjecture were proven on a case by case basis \cite{Askey,GG,Hab,kadell1,zeil-bres1985}.
A uniform proof for $q=1$ was first found by Opdam \cite{opdam} using hypergeometric shift operators.
Eventually, a case-free proof of the full conjecture
was given by Cherednik \cite{cherednik} based on his double affine Hecke algebra.
For more on the extensive literature of Macdonald's constant term conjecture we refer the reader to \cite{FW} and  references therein.

Let $\lambda=(\lambda_0,\dots,\lambda_n)$ be a partition.
In 1988, Macdonald \cite{MacSMC,Mac95} introduced a family of symmetric functions $
P_{\lambda}(q,t)=P_{\lambda}(x_0,\dots,x_n;q,t)$, now called Macdonald polynomials. Given the scalar product on the ring of symmetric functions
in $x_0,\dots,x_n$
\[
\langle f,g \rangle_{q,q^k}:=\frac{1}{(n+1)!}\CT_x
f(x_0,\dots,x_n) g(x_0^{-1},\dots,x_n^{-1}) \prod_{0\leq i\neq j\leq n}(x_i/x_j;q)_k,
\]
Macdonald established the orthogonality
\[
\big\langle P_{\lambda}(q,q^k), P_{\mu}(q,q^k)
\big\rangle_{q,q^k}=0
\quad\text{if $\lambda\neq\mu$},
\]
for $k$ a positive integer.
Moreover, he showed that the quadratic norm
evaluation is given by \cite[page 373]{Mac95}
\[
\big\langle P_{\lambda}(q,q^k), P_{\lambda}(q,q^k)
\big\rangle=\prod_{0\leq i<j\leq n}\frac{(q^{\lambda_i-\lambda_j+1+(j-i)k};q)_{k-1}}
{(q^{\lambda_i-\lambda_j+1+(j-i-1)k};q)_{k-1}}.
\]
For $\lambda=0$ the Macdonald polynomials trivialise
to $1$, so that
\[
\big\langle 1,1\big\rangle_{q,q^k}=
\prod_{i=1}^{n+1}\qbinom{ik-1}{k-1}.
\]
By a simple transformation, it is not difficult to
show that this is equivalent to the
equal parameter case of \eqref{q-Dyson}. This
provides a satisfactory explanation for the
$a_0=a_1=\dots=a_n=k$ case of the $q$-Dyson
constant term identity in terms of orthogonal polynomials.
Finding a similar such explanation for the full
$q$-Dyson identity is an important open problem.

The first step towards a resolution of this
problem was made by Kadell \cite{kadell}, who
formulated an orthogonality conjecture which we will
describe next.
Let $X=(x_0,x_1,\dots)$ be an alphabet of countably many variables.
Then the $r$th complete symmetric function $h_r(X)$
may be defined in terms of its generating function as
\begin{equation}\label{e-gfcomplete}
\sum_{r\geq 0} z^r h_r(X)=\prod_{i\geq 0}
\frac{1}{1-zx_i}.
\end{equation}
More generally, for the complete symmetric function indexed by a composition (or partition)
$v=(v_0,v_1,\dots,v_k)$
\[
h_v:=h_{v_0}\cdots h_{v_k}.
\]
For $a:=(a_0,a_1,\dots,a_n)$ a sequence of non-negative integers, let $x^{(a)}$ denote
the alphabet
\begin{equation}\label{alphabet-x}
x^{(a)}=(x_0,x_0q,\dots,x_0q^{a_0-1},\dots,x_n,x_nq,\dots,x_nq^{a_n-1})
\end{equation}
of cardinality $|a|:=a_0+\cdots+a_n$, and define
the generalized $q$-Dyson constant term
\begin{equation}\label{D-Kadell}
D_{v,\lambda}(a)=\CT_x
x^{-v}h_{\lambda}\big(x^{(a)}\big)
\prod_{0\leq i<j\leq n}
(x_i/x_j;q)_{a_i}(qx_j/x_i;q)_{a_j}.
\end{equation}
Here $v=(v_0,\dots,v_n)\in\mathbb{Z}^{n+1}$,
$x^v$ denotes the monomial $x_0^{v_0}\cdots x_n^{v_n}$
and $\lambda$ is a partition such that
$|v|=|\lambda|$.
(Note that if $|v|\neq|\lambda|$ then
$D_{v,\lambda}(a)=0$.)
For the constant term \eqref{D-Kadell},
Kadell formulated the
following conjecture \cite[Conjecture 4]{kadell}.
\begin{conj}\label{conj-Kadell}
For $r$ a positive integer and $v$ a composition such that $|v|=r$,
\begin{equation}\label{kadellconj}
D_{v,(r)}(a)=
\begin{cases}
\displaystyle
\frac{q^{\sum_{i=k+1}^n a_i}(1-q^{a_k})(q^{|a|};q)_r}{(1-q^{|a|})(q^{|a|-a_k+1};q)_r}
\prod_{i=0}^n\qbinom{a_i+\cdots+a_n}{a_i}
& \text{if $v=(0^{k},r,0^{n-k})$}, \\[6mm]
0 & \text{otherwise}.
\end{cases}
\end{equation}
\end{conj}
In fact Kadell only considered $v=(r,0^n)$ in his conjecture, but the
more general statement given above is what was proved by K\'{a}rolyi,
Lascoux and Warnaar in \cite[Theorem 1.3]{KLW} using multivariable Lagrange
interpolation and key polynomials.
If for a sequence $u=(u_0,\dots,u_n)$ of integers we denote
by $u^{+}$ the sequence obtained from $u$ by ordering the $u_i$
in weakly decreasing order (so that $u^{+}$ is a partition
if $u$ is a composition), then K\'{a}rolyi et al.\ also
proved a closed-form expression for
$D_{v,v^{+}}(a)$ in the case when $v$ is a composition all of whose parts
have multiplicity one, i.e., $v_i\neq v_j$ for all $0\leq i<j\leq n$.
Subsequently, Cai \cite{cai} gave an inductive proof of Kadell's conjecture.
He also showed that the following more general
orthogonality holds.
\begin{thm}\label{thm-Cai}
Let $v\in\mathbb{Z}^{n+1}$ and $\lambda$
a partition such that $|v|=|\lambda|$.
If $D_{v,\lambda}(a)$ is non-vanishing, then
$v^{+}\geq \lambda$ in dominance order.
\end{thm}
We note that the converse of Theorem \ref{thm-Cai} also appears to be true. This is trivially the case for $n=0$, and for $n=1$ we used Maple
to verify that
$D_{v,\lambda}(a)\neq 0$ for
$|v|\leq 23$ and $v^+\geq \lambda$.

In this paper, we are concerned with
the $\lambda=v^+$ case of $D_{v,\lambda}(a)$.
For this case we obtain a recursion for $D_{v,\lambda}(a)$ provided that
the largest part of $v$ occurs with multiplicity one.
Given a sequence $s=(s_0,\dots,s_n)$ and an integer $k$
such that $0\leq k\leq n$, define
$s^{(k)}:=(s_0,\dots,s_{k-1},s_{k+1},\dots,s_n)$.

\begin{theorem}\label{thm-1}
Let $v=(v_0,\dots,v_n)$ be a composition
such that its largest part has multiplicity one in $v$.
Fix a non-negative integer $k$ by $v_k=\max\{v\}$.
Then
\begin{equation}\label{e-Dyson}
D_{v,v^{+}}(a)=q^{\sum_{i=k+1}^n a_i}
\qbinom{v_k+|a|-1}{a_k-1}
D_{v^{(k)},(v^{(k)})^{+}}\big(a^{(k)}\big).
\end{equation}
\end{theorem}
For example, if $v={(0,2,3,2,1)}$, then
$v^+={(3,2,2,1,0)}$, $k=2$ and $v^{(2)}={(0,2,2,1)}$.
If all the non-zero parts
of $v$ have multiplicity one, then we can iterate \eqref{e-Dyson}.
Together with the $q$-Dyson identity \eqref{q-Dyson} this
yields a closed-form formula for
$D_{v,v^{+}}(a)$.

\begin{cor}\label{cor-1}
Let $v=(v_0,\dots,v_n)$ be a composition all of whose
positive parts have multiplicity one,
and set $l:=\ell(v)$,
the number of the non-zero parts of $v$.
Let $\sigma\in\mathfrak{S}_{n+1}$ be any permutation
for which $\sigma(v):=(v_{\sigma(0)},\dots,v_{\sigma(n)})=v^{+}$.
Then
\begin{equation}\label{e-cor1}
D_{v,v^{+}}(a)=
q^{c}
\prod_{i=0}^{l-1}
\qbinom{v_{\sigma(i)}+|a|-a_{\sigma(0)}-\dots-a_{\sigma(i-1)}-1}
{a_{\sigma(i)}-1}
\prod_{i=l}^n
\qbinom{a_{\sigma(i)}+\dots+a_{\sigma(n)}}{a_{\sigma(i)}},
\end{equation}
where
\[
c=\sum_{i=0}^{l-1}\sum_{\substack{j=\sigma(i)+1\\[1pt]
j\notin \{\sigma(0),\dots,\sigma(i-1)\}}}^n a_j.
\]
\end{cor}

Clearly, there are $(n-l+1)!$ admissible permutations
$\sigma\in\mathfrak{S}_{n+1}$.
Since the product
\[
\prod_{i=l}^n
\qbinom{a_{\sigma(i)}+\dots+a_{\sigma(n)}}{a_{\sigma(i)}}
=\frac{(q;q)_{a_{\sigma(l)}+\dots+a_{\sigma(n)}}}{(q;q)_{a_{\sigma(l)}}\cdots (q;q)_{a_{\sigma(n)}}}
\]
is symmetric in $a_{\sigma(l)},\dots,a_{\sigma(n)}$
each such $\sigma$ results in the same expression
for $D_{v,v^{+}}(a)$.
When $l=1$ Corollary~\ref{cor-1} reduces to
the non-zero part of Kadell's (ex)-conjecture,
and when $l=n$ or $l=n+1$ it reduces to a result
by K\'{a}rolyi, Lascoux and Warnaar \cite[Proposition 4.5]{KLW}.

\medskip

The method employed to prove Theorem~\ref{thm-1}
is based on the well-known fact that two polynomials of degree at most $d$ are equal if they are equal at  $d+1$ distinct points.
This method was used previously to prove several
constant term identities, such as in the
Gessel--Xin proof of the $q$-Dyson identity \cite{gess-xin2006} or in the proof of what are
known as first-layer formulas for
$q$-Dyson products \cite{LXZ}.

It is not difficult to show that
$D_{v,v^{+}}(a)$ is a polynomial of degree $v_k+|a|-a_k$
in $q^{a_k}$.
Assuming the conditions of Theorem~\ref{thm-1},
it is also not hard to show that this polynomial
vanishes if
$-a_k\in \{0,1,\dots,v_k+|a|-a_k-1\}$.
However,
since $D_{v,v^{+}}(a)$ is not actually defined for negative integer values of $a_k$, we need to extend
the definition to all integers $a_k$.
For this, we require the theory of iterated Laurent series, developed in \cite{xinresidue}.
In the field of iterated Laurent series,
$D_{v,v^{+}}(a)$ is well-defined for all $a_k\in\mathbb{Z}$ and can
again be viewed as a polynomial in $q^{a_k}$.
To prove the above vanishing properties of $D_{v,v^{+}}(a)$
(again with $v$ as in the theorem),
we combine the Gessel--Xin method with plethystic substitutions,
a powerful tool from the theory of symmetric functions.
It trivially follows that the right-hand side of
\eqref{e-Dyson} satisfies the same polynomiality and
vanishing properties.
By degree considerations, one may conclude that
the left and right-hand sides of \eqref{e-Dyson} are
equal if they agree at one additional point.

The remainder of this paper is organised as follows.
In the next section we introduce some basic
notation used throughout this paper.
In Sections~\ref{sec-ple} and \ref{sec-iterate}
we introduce the two main tools used in this paper---plethystic notation
and substitutions, and iterated Laurent series respectively.
In Section~\ref{sec-proof} we give a proof of Theorem~\ref{thm-1}.

\section{Basic notation}\label{sec-notation}

In this section we introduce some basic notation
used throughout this paper.

For $v=(v_0,v_1,\dots,v_n)$ a sequence, we write
$|v|$ for the sum of its entries, i.e.,
$|v|=v_0+\cdots+v_n$.
Moreover, if $v\in\mathbb{R}^{n+1}$ then we write
$v^{+}$ for the sequence obtained from $v$
by ordering its elements in weakly decreasing order.
If all the entries of $v$ are non-negative integers, we refer to $v$ as a (weak) composition.
A partition is a sequence $\lambda=({\lambda_0,\lambda_1,\dots)}$ of non-negative integers such that
${\lambda_0\geq \lambda_1\geq \cdots}$ and
only finitely-many $\lambda_i$ are positive.
The length of a partition $\lambda$, denoted
$\ell(\lambda)$ is defined to be the number of non-zero $\lambda_i$ (such $\lambda_i$ are known as the parts of $\lambda$).
We adopt the convention of not displaying the
tails of zeros of a partition.
We say that $|\lambda|=\lambda_0+\lambda_1+\cdots$ is the size of the partition $\lambda$.
We adopt the standard dominance order on
the set of partitions of the same size.
If $\lambda,\mu$ are partitions such that $|\lambda|=|\mu|$ then $\lambda\geq \mu$ if
$\lambda_0+\cdots+\lambda_i\geq \mu_0+\cdots+\mu_i$ for all $i\geq 0$.
Similarly, for two integer sequences
$v=(v_0,\dots,v_n)$ and $u=(u_0,\dots,u_m)$, we
write $v\geq u$ if $v_0+\cdots+v_i\geq u_0+\cdots+u_i$
for all $i\geq 0$, where we set $v_i=0$ for $i>n$ and $u_j=0$ for $j>m$.
Note here that we do not require that $|v|=|w|$.
As usual, we write $\lambda>\mu$ if $\lambda\geq \mu$
but $\lambda\neq \mu$, and
$v>u$ if $v\geq u$ but $v\neq u$.

The infinite $q$-shifted factorial is defined as
\[
(z)_{\infty}=(z;q)_{\infty}:=
\prod_{i=0}^{\infty}(1-zq^i),
\]
where, typically, we suppress the base $q$.
Then, for $k$ an integer,
\[
(z)_k=(z;q)_k:=\frac{(z;q)_{\infty}}
{(zq^k;q)_{\infty}}.
\]
Note that
\[
(z)_k=
\begin{cases}
(1-z)(1-zq)\cdots (1-zq^{k-1}) & \text{if $k\geq 0,$}
\\[3mm]
\displaystyle
\frac{1}{(1-zq^k)(1-zq^{k+1})\cdots(1-zq^{-1})}
& \text{if $k<0$.}
\end{cases}
\]
Using the above we can define the
$q$-binomial coefficient as
\[
\qbinom{n}{k}=\frac{(q^{n-k+1})_k}{(q)_k}
\]
for $n$ an arbitrary integer and $k$ a non-negative integer.

\section{Plethystic notation}\label{sec-ple}

Plethystic or $\lambda$-ring notation is a device to facilitate computations in the ring of symmetric functions.
The notion was introduced by Grothendieck \cite{Grothendieck} in
the study of Chern classes.
Nowadays
plethystic notation has become an indispensable computational tool for
organizing and manipulating intricate relationships between symmetric
functions.
In this section, we briefly introduce plethystic notation and substitutions.
For more details, see \cite{haglund,Lascoux,Mellit,RW}.

Denote by $\Lambda_{\mathbb{F}}$ the ring of symmetric functions in countably many variables with coefficients in a field $\mathbb{F}$.
For an alphabet $X=(x_0,x_1,\dots)$,
we additively write $X:=x_0+x_1+\cdots$, and
use plethystic brackets to indicate this additive notation:
\[
f(X)=f(x_0,x_1,\dots)=f[x_0+x_1+\cdots]=f[X], \quad
\text{for $f\in \Lambda_{\mathbb{F}}$.}
\]

For $r$ a positive integer, let $p_r$ be the power sum
symmetric function in the alphabet $X$, defined by
\[
p_r=\sum_{i\geq 0}x_i^r.
\]
In addition, we set $p_0=1$.
For a partition $\lambda=(\lambda_0,\lambda_1,\dots)$, let
\[
p_{\lambda}=p_{\lambda_0}p_{\lambda_1}\cdots.
\]
The $p_r$ are algebraically independent over
$\mathbb{Q}$, and the $p_{\lambda}$ form a basis of $\Lambda_{\mathbb{Q}}$ \cite{Mac95}. That is,
\[
\Lambda_{\mathbb{Q}}=\mathbb{Q}[p_1,p_2,\dots].
\]
Now we introduce a consistent arithmetic on alphabets in terms of the basis
of power sums. In particular,
a power sum whose argument is the sum, difference or Cartesian product
of two alphabets $X$ and $Y$ is defined as
\begin{subequations}\label{asm}
\begin{align}
\label{asm1}
p_r[X+Y]&=p_r[X]+p_r[Y], \\
p_r[X-Y]&=p_r[X]-p_r[Y], \\
p_r[XY]&=p_r[X]p_r[Y].\label{asm3}
\end{align}
\end{subequations}
For example, for the alphabets $X=x_1+x_2+\cdots$ and $Y=y_1+y_2+\cdots$,
the sum of $X$ and $Y$ is $X+Y=x_1+x_2+\cdots+y_1+y_2+\cdots$.
In general we cannot give meaning to division by an arbitrary
alphabet and only division by $1-t$ (the difference of two one-letter
alphabets with ``letters'' $1$ and $t$ respectively) is meaningful.
In particular
\begin{equation}\label{division}
p_r\Big[\frac{X}{1-t}\Big]=\frac{p_r[X]}{1-t^r}.
\end{equation}
Note that the alphabet $1/(1-t)$ may be interpreted as
the infinite alphabet $1+t+t^2+\cdots$.
Indeed, by \eqref{asm1} and \eqref{asm3}
\[
p_r[X(1+t+t^2+\cdots)]=p_r[X]\sum_{k=0}^{\infty} p_r[t^k]=
p_r[X]\sum_{k=0}^{\infty} t^{kr}=\frac{p_r[X]}{1-t^r}.
\]

Having the above rules for plethystic substitutions
we can view symmetric functions as operators acting on alphabets,
and by carrying out complicated substitutions we can turn
simple algebraic identities into much more complicated ones.
For example, since
\[
\sum_{i\leq j} x_ix_j=
\tfrac{1}{2}\Big(\sum_i x_i\Big)^2+
\tfrac{1}{2}\sum_i x_i^2
\]
we have
\[
h_2=\tfrac{1}{2}\big(p_1^2+p_2)
\]
as an identity in the algebra of symmetric functions.
Consequently,
\[
h_2[X]=\tfrac{1}{2}(p_1^2[X]+p_2[X])
\]
where $X$ can be \emph{any} alphabet, obtained by combining the rules
of addition, subtraction, multiplication and division described in
\eqref{asm} and \eqref{division}.

For $r$ a positive integer, let the elementary symmetric
function be defined as
\[
e_r=\sum_{0\leq i_1<\dots<i_r}x_{i_1}x_{i_2}\cdots x_{i_r}. 
\]
Also set $e_0=1$.
By the definition of the elementary symmetric function,
one can observe the following simple fact:
For $r$ a positive integer and $X$ an alphabet of finitely many variables,
\begin{equation}\label{e-vanish}
e_r[X]=0 \qquad \text{if $|X|<r$,}
\end{equation}
where $|X|$ denotes the cardinality of $X$.
This simple fact plays an important role in proving vanishing
properties of expressions of the form $D_{v,\lambda}(a)$.

Finally, we need the following two basic plethystic identities.
One can find proofs in \cite[Theorem 1.27]{haglund}.
\begin{prop}\label{Ple-basic}
Let $X$ and $Y$ be two alphabets. For $r$ a non-negative integer,
\begin{align}\label{e-xy}
h_{r}[X+Y]&=\sum_{i=0}^rh_i[X]h_{r-i}[Y], \\
h_{r}[-X]&=(-1)^re_r[X].\label{e-he}
\end{align}
\end{prop}

\section{Constant term evaluations using iterated Laurent series}\label{sec-iterate}

In this section we introduce some essential ingredients of the field
of iterated Laurent series and describe a basic lemma for extracting
constant terms from rational functions.
Throughout this paper we let $K=\CC(q)$ and work in the field of iterated
Laurent series $K\langle\!\langle x_n, x_{n-1},\dots,x_0\rangle\!\rangle
=K(\!(x_n)\!)(\!(x_{n-1})\!)\cdots (\!(x_0)\!)$, unless specified otherwise.
Elements of $K\langle\!\langle x_n,x_{n-1},\dots,x_0\rangle\!\rangle$
are regarded first as Laurent series in $x_0$, then as
Laurent series in $x_1$, and so on.
The reason the field $K\langle\!\langle x_n, x_{n-1},\dots,x_0\rangle\!\rangle$
is highly suitable for proving constant term identities is
explained in~\cite{gess-xin2006}.
For a more detailed account of the properties of this field,
see \cite{xinresidue} and~\cite{xiniterate}.
Crucial in what is to follow is that the field $K(x_0,\dots,x_n)$ of
rational functions in the variables $x_0,\dots,x_n$ with coefficients in $K$
forms a subfield of $K\langle\!\langle x_n, x_{n-1},\dots,x_0\rangle\!\rangle$, so that every rational function is identified with its unique Laurent
series expansion.

The following series expansion of $1/(1-cx_i/x_j)$
for $c\in K\setminus \{0\}$ forms a key ingredient in our approach:
\[
\frac{1}{1-c x_i/x_j}=
\begin{cases} \displaystyle \sum_{l\geq 0} c^l (x_i/x_j)^l
& \text{if $i<j$}, \\[5mm]
\displaystyle -\sum_{l<0} c^l (x_i/x_j)^l
& \text{if $i>j$}.
\end{cases}
\]
Thus, the constant term in $x_i$ of $1/(1-c x_i/x_j)$
is $1$ if $i<j$ and $0$ if $i>j$.
That is,
\begin{equation}
\label{e-ct}
\CT_{x_i} \frac{1}{1-c x_i/x_j} =
\begin{cases}
    1 & \text{if $i<j$}, \\
    0 & \text{if $i>j$}, \\
\end{cases}
\end{equation}
where, for $f\in K\langle\!\langle x_n, x_{n-1},\dots,x_0\rangle\!\rangle$,
we use the notation $\displaystyle\CT_{x_i} f$ to denote taking the constant term
of $f$ with respect to $x_i$.
An important property of the constant term operators defined this
way is their commutativity:
\[
\CT_{x_i} \CT _{x_j} f = \CT_{x_j} \CT_{x_i} f.
\]
This implies that the operation of taking the constant term in
$K\langle\!\langle x_n,x_{n-1},\dots,x_0\rangle\!\rangle$ is well-defined.

The following lemma is a
basic tool for
extracting constant terms from rational functions and
has appeared previously in~\cite{gess-xin2006}.
\begin{lem}\label{lem-almostprop}
For a positive integer $m$, let $p(x_k)$ be a Laurent polynomial
in $x_k$ of degree at most $m-1$ with coefficients in
$K\langle\!\langle x_n,\dots,x_{k-1},x_{k+1},\dots,x_0\rangle\!\rangle$.
Let $0\leq i_1\leq\dots\leq i_m\leq n$ such that all $i_r\neq k$,
and define
\begin{equation}\label{e-defF}
f=\frac{p(x_k)}{\prod_{r=1}^m (1-x_k/c_r x_{i_r})},
\end{equation}
where $c_1,\dots,c_m\in K\setminus \{0\}$ such that $c_r\neq c_s$ if $x_{i_r}=x_{i_s}$.
Then
\begin{equation}\label{e-almostprop}
\CT_{x_k} f=\sum_{\substack{r=1 \\[1pt] i_r>k}}^m
\big(f\,(1-x_k/c_rx_{i_r})\big)\Big|_{x_k=c_r x_{i_r}}.
\end{equation}
\end{lem}

\section{Proof of Theorem \ref{thm-1}}\label{sec-proof}

To prove Theorem~\ref{thm-1}, which is a recursion for $D_{v,v^{+}}(a)$, we shall first prove a
similar recursion --- see Theorem~\ref{thm-2} below --- for a more general constant term, denoted $D_{v,\lambda}(a,m)$
and defined in \eqref{D-m} below.
As shown in Section~\ref{s-cyclic}, using a cyclic action $\gamma$ on $D_{v,\lambda}(a,m)$, Theorem~\ref{thm-2} implies Theorem~\ref{thm-1}.

\subsection{The constant term $D_{v,\lambda}(a,m)$}\label{s-cyclic}

In this subsection we define $D_{v,\lambda}(a,m)$ mentioned above and show that it suffices to consider this
constant term for those $v=(v_0,\dots,v_n)\in\mathbb{Z}^{n+1}$ for which $\max\{v\}=v_0$.

For $m\in\{0,1,\dots,n+1\}$ and $a=(a_0,\dots,a_n)$ a composition,
define the alphabet
\begin{multline*}
x_m^{(a)}:=(x_0q^{-1},x_0,\dots,x_0 q^{a_0-2}
,\dots,x_{m-1}q^{-1},x_{m-1},\dots,x_{m-1} q^{a_{m-1}-2},  \\
x_m,x_m q,\dots,x_m q^{a_m-1},\dots,x_n,x_n q,\dots,x_n q^{a_n-1}).
\end{multline*}
Note that, plethystically,
\begin{equation}\label{alphabet-m}
x_m^{(a)}=\sum_{i=0}^n \frac{1-q^{a_i}}{1-q}\, x_i q^{-\chi(i<m)},
\end{equation}
where $\chi$ is the truth function.
For $v=(v_0,\dots,v_n)\in \mathbb{Z}^{n+1}$, $\lambda$ a partition such that $|\lambda|=|v|$, and $m$ and $a$ as above,
we define the constant term
\begin{equation}\label{D-m}
D_{v,\lambda}(a,m):=\CT_x
x^{-v}h_{\lambda}\big(x_m^{(a)}\big)
\prod_{0\leq i<j\leq n}
(x_i/x_j)_{a_i}(qx_j/x_i)_{a_j}.
\end{equation}
Clearly, the alphabet $x^{(a)}$ and constant term $D_{v,\lambda}(a)$, defined in \eqref{alphabet-x}
and \eqref{D-Kadell} respectively, are given by
$x^{(a)}=x_0^{(a)}$ and $D_{v,\lambda}(a)=D_{v,\lambda}(a,0)$.
By the homogeneity of the complete symmetric function $h_{\lambda}$ and the fact that
$x^{(a)}_{n+1}=x^{(a)}_0/q$,
\begin{equation}\label{reduce}
D_{v,\lambda}(a,n+1)=q^{-|\lambda|}D_{v,\lambda}(a,0).
\end{equation}
Hence, it suffices to restrict the range of $m$ to $0\leq m\leq n$ or $1\leq m\leq n+1$.

As in \cite{LXZ}, for $f\in K\langle\!\langle x_n, x_{n-1},\dots,x_0\rangle\!\rangle$, define the cyclic action
$\gamma$ by
\[
\gamma \big(f(x_0,x_1,\dots,x_n)\big)=f(x_1,x_2,\dots,x_n,x_0/q).
\]
Then $\displaystyle\CT_x f =\CT_x \gamma(f)$, and, for $0\leq m\leq n$,
\begin{equation}\label{relation-gamma}
\gamma\big(D_{v,\lambda}(a,m)\big)=q^{v_n}D_{\gamma^{-1}(v),\lambda}\big(\gamma^{-1}(a),m+1\big),
\end{equation}
where $\gamma(v):=(v_1,\dots,v_n,v_0)$.
For $k\in \{0,1,\dots,n\}$, by applying \eqref{relation-gamma} exactly $n+1-k$ times and also using \eqref{reduce} we find that
\begin{equation}\label{cycle}
D_{v,\lambda}(a,m)=
\begin{cases}
q^{v_k+\cdots+v_n} D_{\gamma^{-(n+1-k)}(v),\lambda}\big(\gamma^{k-n-1}(a),m'\big) &\text{if $m\leq k$,}\\[3mm]
q^{-v_0-\cdots-v_{k-1}} D_{\gamma^{-(n+1-k)}(v),\lambda}\big(\gamma^{k-n-1}(a),m'\big) &\text{if $m>k$,}
\end{cases}
\end{equation}
where $1\leq m'\leq n+1$ and $m'\equiv m-k\pmod{n+1}$.
In particular, if $k$ is an integer such that $\max\{v\}=v_k$,
then $\gamma^{-(n+1-k)}(v)=(v_k,\dots,v_n,v_0,\dots,v_{k-1})$ has the property that
its first part is its largest part.
Hence \eqref{cycle} allows us to assume without loss of generality that
$v_0=\max\{v\}$.
Also assuming that $v$ is a composition such that $v_0>v_i$ for all $1\leq i\leq n$ we will
prove the following theorem.

\begin{theorem}\label{thm-2}
For $v=(v_0,\dots,v_n)$ a composition such that $v_0=\max\{v\}$ has multiplicity one in $v$,
and $m\in\{1,2,\dots,n+1\}$,
\begin{equation}\label{iterate-m}
D_{v,v^+}(a,m)=q^{\sum_{i=1}^{m-1} a_i-v_0}\qbinom{v_0+|a|-1}{a_0-1}
D_{v^{(0)},(v^{(0)})^+}\big(a^{(0)},m-1\big).
\end{equation}
\end{theorem}

This theorem, together with \eqref{cycle}, implies Theorem~\ref{thm-1} in a few simple steps.

\begin{proof}[Proof of Theorem~\ref{thm-1}]
Let $k\in \{0,1,\dots,n\}$ be fixed by $v_k=\max\{v\}$.
Taking $m=0$ in \eqref{cycle} we have
\begin{equation}\label{e-gammaD}
D_{v,v^+}(a)
=q^{v_k+\cdots+v_n}D_{\gamma^{-(n+1-k)}(v),v^+}\big(\gamma^{k-n-1}(a),n+1-k\big).
\end{equation}
Here $\gamma^{-(n+1-k)}(v)$ has the property that
its first part, $v_k$, is its unique largest part.
Thus we can apply \eqref{iterate-m} to obtain
\begin{multline*}
D_{\gamma^{-(n+1-k)}(v),v^+}\big(\gamma^{k-n-1}(a),n+1-k\big)  \\
=q^{\sum_{i=k+1}^na_i-v_k}\qbinom{v_k+|a|-1}{a_k-1}
D_{\gamma^{-(n-k)}(v^{(k)}),(v^{(k)})^+}\big(\gamma^{k-n}(a^{(k)}),n-k\big),
\end{multline*}
where $\gamma^{-(n-k)}(v^{(k)})=(v_{k+1},\dots,v_n,v_0,\dots,v_{k-1})$
and $\gamma^{k-n}(a^{(k)})=(a_{k+1},\dots,a_n,a_0,\dots,a_{k-1})$.
Using \eqref{e-gammaD} by taking $(k,v,a)\mapsto (k+1,v^{(k)},a^{(k)})$, we have
\begin{align*}
D_{v^{(k)},(v^{(k)})^+}(a^{(k)})
&=q^{\sum_{i=k+1}^n v_i}D_{\gamma^{-(n-k)}(v^{(k)}),(v^{(k)})^+}\big(\gamma^{k-n}(a^{(k)}),n-k\big).
\end{align*}
As a result,
\[
D_{\gamma^{-(n+1-k)}(v),v^+}\big(\gamma^{k-n-1}(a),n+1-k\big)
=q^{\sum_{i=k+1}^na_i-\sum_{i=k}^n v_i}\qbinom{v_k+|a|-1}{a_k-1}
D_{v^{(k)},(v^{(k)})^+}(a^{(k)}).
\]
Substituting this into \eqref{e-gammaD}, we finally obtain
\[
D_{v,v^{+}}(a)=q^{\sum_{i=k+1}^n a_i}
\qbinom{v_k+|a|-1}{a_k-1}
D_{v^{(k)},(v^{(k)})^{+}}\big(a^{(k)}\big),
\]
completing the proof.
\end{proof}

\subsection{Outline of the proof of Theorem~\ref{thm-2}}

Our proof of Theorem~\ref{thm-2} is quite lengthy and involved, and before
presenting the full details we briefly outline the three key steps.\\

\begin{enumerate}
\item \textbf{Polynomiality} --- We will show that,
for fixed non-negative integers $a_1,\dots,a_n$, the constant term
$D_{v,\lambda}(a,m)$ is a polynomial in $q^{a_0}$ of degree at most
$a_1+\cdots+a_n+v_0$.\\

\item \textbf{Determination of roots} --- We will show that
$D_{v,v^+}(a,m)$ vanishes for $-a_0\in\{0,1,\dots,a_1+\cdots+a_n+v_0-1\}$
if $v=(v_0,\dots,v_n)$ is a composition such that $v_0=\max\{v\}$ has multiplicity one in $v$.\\

\item \textbf{Value at $\boldsymbol{a_0=1}$} --- Assuming the conditions of Theorem~\ref{thm-2},
we will show that $D_{v,v^+}(a,m)$
evaluated at $a_0=1$ can be  expressed as the same constant term with $(n,m)\mapsto (n-1,m-1)$. That is
\[
D_{v,v^+}(a,m)|_{a_0=1}=
q^{\sum_{i=1}^{m-1} a_i-v_0}D_{v^{(0)},(v^{(0)})^+}\big(a^{(0)},m-1\big).
\]
\end{enumerate}

The details of these key steps will be presented in the subsections \ref{sec3}, \ref{sec6} and \ref{sec7}
respectively. Subsection \ref{s-blem} prepares some technical preliminaries needed in Subsection~\ref{sec6}.

\subsection{Polynomiality}\label{sec3}

As mentioned above, the aim of this subsection is to prove that the constant term $D_{v,\lambda}(a,m)$ is a polynomial in $q^{a_0}$ of degree at most $a_1+\cdots+a_n+v_0$.

We begin by recalling \cite[Lemma 2.2]{LXZ}.

\begin{lem}\label{lem1}
Let $L(x_1,\dots,x_n)$ be an arbitrary Laurent polynomial.
Then, for fixed non-negative integers $a_1,\dots,a_n$,
and $t$ an integer not exceeding $a_1+\cdots+a_n$,
\begin{equation}\label{p1}
\CT_x x_0^{t} L(x_1,\dots,x_n) \prod_{0\leq i<j\leq n} (x_i/x_j)_{a_i}(qx_j/x_i)_{a_j}
\end{equation}
is a polynomial in $q^{a_0}$
of degree at most $a_1+\cdots+a_n-t$.
Moreover, if $t>a_1+\cdots+a_n$, then the constant term \eqref{p1} vanishes.
\end{lem}

We remark that the correct interpretation of the above lemma is that, for
$t\leq a_1+\cdots+a_n$, there exists a polynomial $P(x)$ of degree at most $a_1+\cdots+a_n-t$ such that,
for all non-negative integers $a_0$,
\[
\CT_x x_0^{t} L(x_1,\dots,x_n) \prod_{0\leq i<j\leq n} (x_i/x_j)_{a_i}(qx_j/x_i)_{a_j}=P(q^{a_0}).
\]

Using Lemma~\ref{lem1} it is not hard to show that the constant term $D_{v,\lambda}(a,m)$ is a polynomial in $q^{a_0}$
for fixed non-negative integers $a_1,\dots,a_n$. This is the content of the next proposition.

\begin{prop}\label{cor-poly}
Let $a_1,\dots,a_n$ be fixed non-negative integers and $D_{v,\lambda}(a,m)$ be defined as in \eqref{D-m}.
If $-v_0\leq a_1+\cdots+a_n$, then $D_{v,\lambda}(a,m)$
is a polynomial in $q^{a_0}$ of degree at most $a_1+\cdots+a_n+v_0$.
If $-v_0>a_1+\cdots+a_n$ then $D_{v,\lambda}(a,m)=0$.
\end{prop}

\begin{proof}
We write
\[
x_m^{(a)}=\frac{1-q^{a_0}}{1-q}\, x_0 q^{-\chi(0<m)}+\hat{x}_m^{(a)},
\]
where
\[
\hat{x}_m^{(a)}:=\sum_{i=1}^n \frac{1-q^{a_i}}{1-q}\, x_i q^{-\chi(i<m)}.
\]
Then, by repeated use of \eqref{e-xy} with
\[
X\mapsto \frac{1-q^{a_0}}{1-q}\, x_0 q^{-\chi(0<m)}
\quad\text{and}\quad
Y\mapsto  \hat{x}_m^{(a)}
\]
and the homogeneity of the complete symmetric function,
the constant $D_{v,\lambda}(a,m)$ can be expanded as
\begin{multline}\label{e-extract1}
D_{v,\lambda}(a,m)
=\sum_k q^{-\chi(0<m)|k|}
h_k\Big[\frac{1-q^{a_0}}{1-q}\Big] \\ \times
\CT_x \frac{x_0^{|k|-v_0}}{x_1^{v_1}\cdots x_n^{v_n}}h_{\lambda-k}\big[\hat{x}_m^{(a)}\big]
\prod_{0\leq i<j\leq n} (x_i/x_j)_{a_i}(qx_j/x_i)_{a_j},
\end{multline}
where $k:=(k_0,\dots,k_{\ell(\lambda)-1})$ is a composition and
the sum is over $0\leq k_i\leq \lambda_i$ for $0\leq i\leq \ell(\lambda)-1$.
Note that, generally, $k$ and $\lambda-k$ are compositions, not partitions.
By Lemma~\ref{lem1} the constant term in \eqref{e-extract1} vanishes if $a_1+\cdots+a_n+v_0<|k|$,
and is a polynomial in $q^{a_0}$ of degree at most $a_1+\cdots+a_n+v_0-|k|$ if $a_1+\cdots+a_n+v_0\geq |k|$.
Together with the fact that $h_k[(1-z)/(1-q)]$ is a polynomial in $z$ of degree $|k|$\footnote{It is easily
shown that $h_r[(1-z)/(1-q)]=(z)_r/(q)_r$, see e.g., \cite[page 27]{Mac95}.},
each summand in \eqref{e-extract1} is either a polynomial in $q^{a_0}$ of degree
at most $a_1+\cdots+a_n+v_0$ or is $0$.
Moreover, if $a_1+\cdots+a_n+v_0<0$, then every constant term in
\eqref{e-extract1} vanishes and $D_{v,\lambda}(a,m)=0$.
\end{proof}

\subsection{Preliminaries for the determination of the roots of $D_{v,\lambda}(a,m)$}\label{s-blem}

In this subsection we prepare some general results used in the next section to determine the roots of
$D_{v,\lambda}(a,m)$.

\begin{lem}\label{lem-import}
For $s$ a positive integer, let $(b_{1},\dots,b_{s+1})$ and $(k_1,\dots,k_s)$ be compositions
such that $1\leq k_{i}\leq b_{1}+\dots+b_{s+1}$ for $1\leq i\leq s$.
Then at least one of the following holds:
\begin{enumerate}
\item $1\leq k_{i}\leq b_{i}$ for some $i$ with $1\leq i\leq s$;
\item $-b_{j}\leq k_{i}-k_{j}\leq b_{i}-1$ for some $1\leq i<j\leq s$;
\item there exists a permutation $w\in\mathfrak{S}_s$ and a composition $(t_1,\dots,t_s)$
such that
\begin{equation}\label{e-k1}
k_{w(j)}-k_{w(j-1)}=b_{w(j)}+t_j \quad \text{for $1\leq j\leq s$.}
\end{equation}
Here $k_0=w(0):=0$, the $t_i$ satisfy $\sum_{j=1}^{s}t_j\leq b_{s+1}$ and $t_j>0$
if $w(j-1)<w(j)$ for $1\leq j\leq s$.
\end{enumerate}
\end{lem}

When $b_{s+1}=0$ case (3) can not occur.
Indeed, if (3) were to hold with $b_{s+1}=0$ this would imply $\sum_{j=1}^s t_j\leq 0$,
contradicting the fact that $t_1>0$.
This special case of the lemma corresponds to \cite[Lemma 4.2]{gess-xin2006} by Gessel and Xin.
If (3) holds with $b_{s+1}=1$, then $t_1=1$ and $t_j=0$ for $2\leq j\leq s$.
Hence $w(j-1)>w(j)$ for $2\leq j\leq s$ so that $w=(s,\dots,2,1)$ and
$k_i=1+\sum_{j=i}^s b_j$ for all $i$.
This special case of the lemma appeared previously as \cite[Lemma 3.2]{LXZ}.
We finally remark that (1) and (3) can not hold simultaneously.
If (3) were to hold, then by \eqref{e-k1} we have $k_{w(j)}\geq b_{w(j)}+1$ for all $j$, contradicting to (1). Also, it is not hard to show that (2) and (3) can not hold simultaneously.
\begin{proof}
We prove the lemma by showing that if (1) and (2) fail then (3) must hold.

Assume that (1) and (2) are both false.
Then we construct a weighted tournament $T$ on the complete graph
on $s$ vertices, labelled $1,\dots,s$, as follows.
For the edge $(i,j)$ with $1\leq i<j\leq s$ we draw an arrow
from $j$ to $i$ and attach a weight $b_i$ if $k_i-k_j\ge b_i$.
If, on the other hand, $k_i-k_j\le -b_j-1$ then we draw an arrow
from $i$ to $j$ and attach the weight $b_j+1$.
Note that the weight of each edge of a tournament is non-negative.

We call a directed edge from $i$ to $j$ ascending if $i<j$.
It is immediate from our construction that
(i) the weight of the edge $i\to j$ is less than or equal $k_j-k_i$, and
(ii) the weight of an ascending edge is positive.

We will use (i) and (ii) to show that any of the above-constructed
tournaments is acyclic and hence transitive.
As consequence of (i), the weight of a directed path from $i$ to $j$ in $T$,
defined as the sum of the weights of its edges, is at most $k_j-k_i$.
Proceeding by contradiction, assume that $T$ contains a cycle $C$.
By the above, the weight of $C$ must be non-positive, and hence $0$.
Since $C$ must have at least one ascending edge, which by (ii) has positive
weight, the weight of $C$ is positive, a contradiction.

Since each $T$ is transitive, there is exactly one directed Hamilton path $P$ in $T$,
corresponding to a total order of the vertices.
Assume $P$ is given by
\[
P=w(1)\rightarrow w(2)\rightarrow\cdots\rightarrow w(s-1)\rightarrow w(s),
\]
where we have suppressed the edge weights.
Then $k_{w(s)}-k_{w(1)}\ge b_{w(2)}+\dots+b_{w(s)}$, and thus
\begin{align}\label{e-contradiction}
k_{w(s)}&\ge k_{w(1)}+b_{w(2)}+\dots+b_{w(s)}\nonumber\\
    & \ge b_{w(1)}+1+b_{w(2)}+\dots+b_{w(s)}\nonumber\\
    &=b_1+\dots+b_s+1.
\end{align}
Together with the assumption that $k_{w(s)}\leq b_1+\dots+b_{s+1}$
this implies that $P$ has at most $b_{s+1}-1$ ascending edges.
Let $(t_1,\dots,t_s)$ be a composition such that \eqref{e-k1} holds.
When $j=1$ this gives $k_{w(1)}=b_{w(1)}+t_1$.
Since (1) does not hold, $k_{w(1)}\geq b_{w(1)}+1$, so that $t_1>0$.
For $2\leq j\leq s$, if $w(j-1)\to w(j)$ is an ascending edge, then $t_j$ is a positive integer.
That is, for $2\leq j\leq s$ if $w(j-1)<w(j)$ then $t_j>0$.
Since
\begin{align*}
\sum_{j=1}^s (k_{w(j)}-k_{w(j-1)})=k_{w(s)}
&=b_{w(1)}+\dots+b_{w(s)}+t_1+\dots+t_s \\
&=b_1+\dots+b_s+t_1+\dots+t_s
\leq b_1+\dots+b_s+b_{s+1},
\end{align*}
we have $t_1+\dots+t_s\leq b_{s+1}$.
This completes the proof of the assertion that (3) must hold if
both (1) and (2) fail.
\end{proof}

Our next proposition concerns alphabets of the form $x_m^{(a)}$ as defined in \eqref{alphabet-m}.

\begin{prop}\label{lem-subs}
For $s$ a positive integer,
let $(b_1,\dots,b_{s+1})$ and $(k_1,\dots,k_s)$ be compositions such that $1\leq k_i\leq b_1+\dots+b_{s+1}$ for $1\leq i\leq s$.
If the $k_i$ are such that (3) of Lemma~\ref{lem-import} holds, then for $m$ a non-negative integer
\begin{equation}\label{e-subs}
-\sum_{i=0}^s \frac{1-q^{b_i}}{1-q}\, x_iq^{-\chi(i<m)}
\bigg|_{\substack{b_0=-\sum_{i=1}^{s+1} b_i, \\[1pt] x_i=q^{k_s-k_i},\,0\leq i\leq s}}
=q^{n_1}+\dots+q^{n_{b_{s+1}}},
\end{equation}
where $\{n_1,\dots,n_{b_{s+1}}\}$ is a set of integers determined by $m$ and the $b_i$ and $k_j$.
\end{prop}

We remark that the set $\{n_1,\dots,n_{b_{s+1}}\}$ can be explicitly determined.
However, since the precise values of the $n_i$ play no role in the following, we
have omitted them from the above statement.
Indeed, the important fact about the right-hand side is that, viewed as an
alphabet, has cardinality $b_{s+1}$.

\begin{proof}
Denote the left-hand side of \eqref{e-subs} by $L$.
Carrying out the substitutions
\[
b_0\mapsto -\sum_{i=1}^{s+1} b_i\quad\text{and}\quad  x_i\mapsto q^{k_s-k_i} \text{ for $0\leq i\leq s$}
\]
in
\[
-\sum_{i=0}^s \frac{1-q^{b_i}}{1-q}\, x_iq^{-\chi(i<m)},
\]
we obtain
\begin{align*}\label{s1}
L&=
-\frac{q^{k_s}}{1-q}\bigg(\big(1-q^{-\sum_{i=1}^{s+1} b_i}\big)
q^{-\chi(0<m)}+\sum_{i=1}^s(1-q^{b_i})q^{-k_i-\chi(i<m)}\bigg)\nonumber \\
&=-\frac{q^{k_s}}{1-q}\bigg(\big(1-q^{-\sum_{i=1}^{s+1} b_i}\big)q^{-\chi(0<m)}+
\sum_{i=1}^s(1-q^{b_{w(i)}})q^{-k_{w(i)}-\chi(w(i)<m)}\bigg),
\end{align*}
where $w\in \mathfrak{S}_s$ is any permutation such that \eqref{e-k1} holds.
By summing that equation over $j$ from $1$ to $i$, we find
\[
k_{w(i)}=\sum_{j=1}^i (b_{w(j)}+t_j) \quad \text{for $1\leq i\leq s$.}
\]
Hence
\[
L=-\frac{q^{k_s}}{1-q}\bigg( \big(1-q^{-\sum_{i=1}^{s+1} b_i}\big)q^{-\chi(0<m)}+
\sum_{i=1}^s(1-q^{b_{w(i)}})
q^{-\sum_{j=1}^i (b_{w(j)}+t_j)-\chi(w(i)<m)}\bigg).
\]
By rearranging the terms in the above expression this may be written as
\begin{multline*}
L=\frac{q^{k_s}}{1-q}\bigg(q^{-\sum_{i=1}^{s+1} b_i-\chi(0<m)}
\Big(1-q^{b_{s+1}-\sum_{j=1}^st_j+\chi(0<m)-\chi(w(s)<m)}\Big) \\
+\sum_{i=1}^s q^{b_{w(i)}-\sum_{j=1}^i (b_{w(j)}+t_j)-\chi(w(i)<m)}
\Big(1-q^{t_{i}-\chi(w(i-1)<m)+\chi(w(i)<m)}\Big)\bigg).
\end{multline*}
Next we will show that
\begin{subequations}
\begin{equation}\label{s4}
b_{s+1}-\sum_{j=1}^s t_j+\chi(0<m)-\chi(w(s)<m)\in\mathbb{N}
\end{equation}
and
\begin{equation}\label{s3}
t_i-\chi(w(i-1)<m)+\chi(w(i)<m)\in\mathbb{N} \quad\text{for $1\leq i\leq s$},
\end{equation}
\end{subequations}
where $\mathbb{N}=\{0,1,2,\dots\}$.
Since $w(s)>0$, it is clear that $\chi(0<m)\geq \chi(w(s)<m)$.
Together with the condition $\sum_{j=1}^{s}t_j\leq b_{s+1}$, this implies that \eqref{s4} holds.
To also show that \eqref{s3} holds, it suffices to show that
if $\chi(w(i)<m)=0$ and $\chi(w(i-1)<m)=1$ for some $i$, then $t_i$ is a positive integer.
If $\chi(w(i)<m)=0$ and $\chi(w(i-1)<m)=1$, then $w(i)\geq m$ and $w(i-1)<m$ respectively.
Hence $w(i-1)<w(i)$.
It follows that $t_i>0$ by the conditions on the $t_i$ in item (3) of Lemma~\ref{lem-import}.
Consequently, \eqref{s3} holds as well.
Since \eqref{s4} and \eqref{s3} hold,
and $(1-q^n)/(1-q)=1+\dots+q^{n-1}$ for $n\in\mathbb{N}$, we may conclude
that $L=q^{n_1}+\dots+q^{n_p}$
where $p$ is given by
\begin{align*}
p&=b_{s+1}-\sum_{j=1}^st_j+\chi(0<m)-\chi(w(s)<m)
+\sum_{i=1}^s \Big(t_i-\chi(w(i-1)<m)+\chi(w(i)<m)\Big) \\
&=b_{s+1},
\end{align*}
completing the proof.
\end{proof}

\subsection{Determination of the roots of $D_{v,v^+}(a,m)$}\label{sec6}

In this subsection, we will determine all the roots of $D_{v,v^+}(a,m)$ if $v$ is a composition and $v_0=\max\{v\}$ has multiplicity one in $v$.
More precisely, $D_{v,v^+}(a,m)$ vanishes for $-a_0\in \{0,1,\dots,a_1+\cdots+a_n+v_0-1\}$.

Since $D_{v,\lambda}(a,m)$ is a polynomial in $q^{a_0}$ by Proposition~\ref{cor-poly}, we can extend the definition of $a_0$ to all integers. In this subsection, we are concerned with $D_{v,v^+}(a,m)$ for $a_0$ a non-positive integer. Thus, for simplicity we denote the Laurent series of $D_{v,v^+}(a,m)$ by $Q(-a_0)$ if $a_0$ is a non-positive integer.
That is, for $d$ a non-negative integer
\begin{multline}\label{qh}
Q(d)=x^{-v}h_{v^+}\bigg[\frac{1-q^{-d}}{1-q}x_0q^{-\chi(0<m)}+\sum_{i=1}^n
\frac{1-q^{a_i}}{1-q}x_iq^{-\chi(i<m)}\bigg] \\
\times \prod_{i=1}^n \frac{(qx_i/x_0)_{a_i}}{(q^{-d}x_0/x_i)_d}
\prod_{1\leq i<j\leq n} (x_i/x_j)_{a_i}(qx_j/x_i)_{a_j},
\end{multline}
and
\[
\CT_{x} Q(d)=D_{v,v^+}(a,m)|_{a_0=-d}.
\]
Therefore, instead of determining the roots of $D_{v,v^+}(a,m)$,
we prove
\[
\CT_x Q(d)=0\quad \text{for $d\in \{0,1,\dots,a_1+\cdots+a_n+v_0-1\}.$}
\]
Here and in the following of this subsection, we assume that $v$ is a composition, $v_0=\max\{v\}$ has multiplicity one in $v$, and $m\in\{0,1,\dots,n+1\}$,
unless specified otherwise. Furthermore, we assume $n$ is a positive integer, since the $n=0$ case for $Q(d)$ is trivial.

We begin by showing that $\CT_x Q(0)=0$.
Since $v_0$ is the unique largest part of $v$, it is a positive integer.
By a degree consideration in $x_0$ of $Q(0)$,
it is easy to see that $\CT_x Q(0)=0$.
In the remainder of this subsection,
we will prove $\CT_x Q(d)=0$ for
$d\in \{1,\dots,a_1+\cdots+a_n+v_0-1\}$ by combining the Gessel--Xin method with plethystic substitutions.
The main process of the Gessel--Xin method is to recursively apply Lemma~\ref{lem-almostprop} to a rational function of the form \eqref{e-defF} to extract the constant term in one variable each time,
until eliminating all the variables of the rational function.

To apply Lemma~\ref{lem-almostprop} to $Q(d)$,
we need to show that $Q(d)$ is of the form \eqref{e-defF} with respect to $x_0$.
The denominator of $Q(d)$ --- $\prod_{i=1}^n (q^{-d}x_0/x_i)_d$ ---
is of the form
\[
\prod_{r=1}^{nd} (1-x_0/c_r x_{i_r}),
\]
which has degree $nd$ in $x_0$.
Here $c_1,\dots,c_{nd}\in K\setminus \{0\}$ satisfy $c_r\neq c_s$ if $x_{i_r}=x_{i_s}$.
To get the degree in $x_0$ of the numerator of $Q(d)$,
we need the next result.
\begin{prop}\label{prop-hr}
Let $d$ and $r$ be non-negative integers, and $\{n_1,\dots,n_d\}$ be a set of integers.
For $z$ an arbitrary letter and $Y$ an alphabet independent of $z$,
\begin{equation}\label{form-h}
h_r\big[{-}(q^{n_1}+\cdots+q^{n_d})z+Y\big]
\end{equation}
is a polynomial in $z$ of degree at most $\min\{r,d\}$.
In particular, if $Y=0$ and $d<r$ then \eqref{form-h} vanishes.
\end{prop}
\begin{proof}
By \eqref{e-xy} we can expand \eqref{form-h} as
\[
\sum_{i=0}^rz^ih_i\big[{-}(q^{n_1}+\cdots+q^{n_d})\big]h_{r-i}[Y].
\]
By \eqref{e-he} it becomes
\[
\sum_{i=0}^r(-z)^ie_i\big[q^{n_1}+\cdots+q^{n_d}\big]h_{r-i}[Y].
\]
Since $e_i\big[q^{n_1}+\cdots+q^{n_d}\big]=0$ for $i>d$ by \eqref{e-vanish},
the above sum reduces to
\[
\sum_{i=0}^{\min\{r,d\}}(-z)^ie_i\big[q^{n_1}+\cdots+q^{n_d}\big]h_{r-i}[Y],
\]
which is a polynomial in $z$ of degree at most $\min\{r,d\}$.

If $Y=0$ then \eqref{form-h} becomes
\[
h_r\big[{-}(q^{n_1}+\cdots+q^{n_d})z\big]
=(-z)^re_r\big[q^{n_1}+\cdots+q^{n_d}\big].
\]
The above equation holds by \eqref{e-he}. It vanishes for $d<r$ by \eqref{e-vanish}.
\end{proof}
For $d$ a positive integer
\[
\frac{1-q^{-d}}{1-q}
=-(q^{-1}+q^{-2}+\cdots+q^{-d}).
\]
Thus, for $r$ a non-negative integer
\[
h_r\bigg[\frac{1-q^{-d}}{1-q}x_0q^{-\chi(0<m)}+\sum_{i=1}^n
\frac{1-q^{a_i}}{1-q}x_iq^{-\chi(i<m)}\bigg]
\]
is of the form \eqref{form-h} with
\[
x_0\mapsto z, \quad
\sum_{i=1}^n\frac{1-q^{a_i}}{1-q}x_iq^{-\chi(i<m)}\mapsto Y,\quad
n_i=-i-\chi(0<m)\ \text{for $i=1,\dots,d$}.
\]
By Proposition~\ref{prop-hr} it is a polynomial in $x_0$ of degree at most $\min\{r,d\}$.
It follows that the Laurent polynomial in $x_0$ of the numerator of $Q(d)$
\[
x^{-v}h_{v^+}\bigg[\frac{1-q^{-d}}{1-q}x_0q^{-\chi(0<m)}+\sum_{i=1}^n
\frac{1-q^{a_i}}{1-q}x_iq^{-\chi(i<m)}\bigg]
\prod_{i=1}^n (qx_i/x_0)_{a_i} \cdot D
\]
has degree
\[
\sum_{i=0}^{n}\min \{v_i,d\}-v_0
\]
in $x_0$.
Here
\[
D:=\prod_{1\leq i<j\leq n} (x_i/x_j)_{a_i}(qx_j/x_i)_{a_j}
\]
is independent of $x_0$.
For $n$ a positive integer the degree
\[
\sum_{i=0}^{n}\min \{v_i,d\}-v_0\leq v_1-v_0+nd<nd.
\]
The last inequality holds because $v_1<v_0$ by the fact that $v_0$ is the unique largest part of $v$.
The above inequality shows that the degree in $x_0$ of the numerator of $Q(d)$ --- $\sum_{i=0}^{n}\min \{v_i,d\}-v_0$ --- is strictly less than $nd$,
the degree in $x_0$ of the denominator of $Q(d)$.
Therefore, $Q(d)$ is of the form \eqref{e-defF}.
Then, by applying Lemma \ref{lem-almostprop} to $Q(d)$ with respect to $x_0$ we obtain
\begin{equation}\label{qd}
\CT_{x_0}Q(d)=\sum_{\substack{0<u_1\leq n,\\ 1\leq k_1\leq d}}Q(d\Mid u_1;k_1),
\end{equation}
where
\[
Q(d\Mid u_1;k_1)=Q(d)\Big(1-\frac{x_0}{x_{u_1}q^{k_1}}\Big)\bigg|_{x_0=x_{u_1}q^{k_1}}.
\]

For each term in \eqref{qd} we  extract the constant term in $x_{u_1}$, and then perform further constant term extractions,
eliminating one variable at each step.
In order to keep track of the
terms we obtain, we introduce some notation from
\cite{gess-xin2006}.

Let $f$ be a rational function of $x_0,\dots,x_n$. For $s$ a positive integer, let ${k}=(k_1,\dots,k_s)$ and ${u}=(u_1,\dots,u_s)$ be compositions such that
$0<u_1<\cdots<u_s\leq n$. Define $E_{u,k}f$ to be
the result of replacing $x_{u_i}$ in $f$ with
$x_{u_s}q^{k_s-k_i}$ for $i=0,\dots,s-1$, where
$u_0=k_0:=0$. Then, for $d$ a positive integer and
$0<k_i\leq d$,
\begin{equation}\label{qh3}
Q(d\Mid u;k):=Q(d\Mid u_1,\dots,u_s;k_1,\dots,k_s)=
E_{u,k}\bigg(Q(d)\prod_{i=1}^{s}\Big(1-\frac{x_0}{x_{u_i}q^{k_i}}\Big)\bigg).
\end{equation}
Note that the product on the right-hand side of \eqref{qh3} cancels
all the factors in the denominator of $Q(d)$ that would be taken to
zero by $E_{u,k}$.

\begin{lem}\label{lem-lead1}
Let $v$ be a composition such that $v_0$ is its unique largest part, $a_1,\dots,a_n$ be non-negative integers and $m\in\{0,1,\dots,n+1\}$.
For $d\in \{1,\dots,a_1+\cdots+a_n+v_0-1\}$, the rational functions $Q(d\Mid u;k)$
defined as in \eqref{qh3} have the following properties:
\begin{itemize}
\item[(i)] If $1\leq k_i\leq a_{u_1}+\cdots+a_{u_s}$ for all $i$ with
$1\leq i\leq s\leq n$,
then $Q(d\Mid u;k)=0$.
\item[(ii)] If $k_i>a_{u_1}+\cdots+a_{u_s}$ for some $i\in \{1,\dots,s\}$,
then
\begin{equation}\label{qh4}
\CT_{x_{u_s}}Q(d\Mid u;k)=
\begin{cases}
\displaystyle \sum_{\substack{u_s<u_{s+1}\leq n,\\
1\leq k_{s+1}\leq d}}
Q(d\Mid u_1,\dots,u_s,u_{s+1};k_1,\dots,k_s,k_{s+1})\ &\text{for $u_s<n$;}\\
0 &\text{for $u_s=n$}.
\end{cases}
\end{equation}
\end{itemize}
In particular,
\begin{equation}\label{e-Qn}
\CT_{x_n} Q(d\Mid 1,\dots,n;k_1,\dots,k_n)=0.
\end{equation}
\end{lem}
\begin{proof}
[Proof of \rm{(i)}]
Taking $b_i\mapsto a_{u_i}$ for $i=1,\dots,s$ and $b_{s+1}=0$ in Lemma~\ref{lem-import}, we have the following result.
If $1\leq k_i\leq a_{u_1}+\cdots+a_{u_s}$ for $i=1,\dots,s$,
then either
$1\le k_i\le a_{u_i}$ for some $i$, or $-a_{u_j}\le
k_i-k_j\le a_{u_i}-1$ for some $i<j$. If $1\le k_i\le a_{u_i}$ for some $i$,
then $Q(d\Mid u;k)$ has the factor
$$E_{u,k}
\Big((qx_{u_i}/x_{0})_{a_{u_i}} \Big)
=\lrq{x_{u_s}q^{k_s-k_i}}{x_{u_s}q^{k_s}}{q}_{\!\!a_{u_i}}
=(q^{1-k_i})_{a_{u_i}}=0.$$

If $-a_{u_j}\le k_i-k_j\le a_{u_i}-1$ for some $i<j$, then $Q(d\Mid u;k)$ has
the factor
$$E_{u,k} \Big((x_{u_i}/x_{u_j})_{a_{u_i}}(qx_{u_j}/x_{u_i})_{a_{u_j}}\Big)
=E_{u,k} \Big(q^{\binom{a_{u_j}+1}{2}} (-x_{u_j}/x_{u_i})^{a_{u_j}}
(q^{-a_{u_j}}x_{u_i}/{x_{u_j}})_{a_{u_i}+a_{u_j}}
\Big),$$
which is equal to
$$q^{\binom{a_{u_j}+1}{2}}
(-q^{k_i-k_j})^{a_{u_j}}(q^{k_j-k_i-a_{u_j}})_{a_{u_i}+a_{u_j}}=0.
$$
\smallskip
\noindent\emph{Proof of \rm{(ii)}.}
Note that since $d\ge k_i$
for all $i$, the hypothesis implies that
$d>a_{u_1}+\cdots+a_{u_s}$.
Let
\begin{equation}\label{defi-b}
d=\sum_{i=1}^sa_{u_i}+b
\end{equation}
for a positive integer $b$. Then $1\leq k_i\leq \sum_{i=1}^sa_{u_i}+b$
for all $i=1,\dots,s$.
If we take $b_i\mapsto a_{u_i}$ for $1\leq i\leq s$ and $b_{s+1}\mapsto b$ in Lemma~\ref{lem-import}, then at least one of the following three cases holds:
\begin{enumerate}
\item $1\leq k_{i}\leq a_{u_i}$ for some $i$ with $1\leq i\leq s$;
\item $-a_{u_j}\leq k_{i}-k_{j}\leq a_{u_i}-1$ for some $1\leq i<j\leq s$;
\item there exists a permutation $w\in\mathfrak{S}_s$ and a composition $(t_1,\dots,t_s)$
such that
\[
k_{w(j)}-k_{w(j-1)}=a_{u_{w(j)}}+t_j \quad \text{for $1\leq j\leq s$.}
\]
Here $k_0=w(0):=0$, the $t_i$ satisfy $\sum_{j=1}^{s}t_j\leq b$ and $t_j>0$
if $w(j-1)<w(j)$ for $1\leq j\leq s$.
\end{enumerate}
If either (1) or (2) holds, then $Q(d\Mid u;k)=0$ for $1\leq s\leq n$ by the same argument as that in part (i). In addition, \eqref{qh4} holds if the $k_i$ satisfy (1) or (2) since both sides vanish.
It remains to show that \eqref{qh4} holds if the $k_i$ satisfy (3).
We discuss this according to the following three cases:
(a) $s=n$; (b) $1\leq s<n$ and $u_s=n$; (c) $1\leq s<n$ and $u_s<n$.

If $s=n$, then $u_i=i$ for $i=1,\dots,n$. In this case, we prove \eqref{qh4} by showing that
\begin{equation}\label{qh5}
\CT_{x_n} Q(d\Mid 1,\dots,n;k_1,\dots,k_n)=0
\end{equation}
for the $k_i$ satisfy (3) and $d>a_1+\cdots+a_n$.
By \eqref{defi-b} if $s=n$ then
$b=d-a_1-\cdots-a_n$.
Together with $d<a_1+\cdots+a_n+v_0$ yields $b<v_0$ for $s=n$.
If the $k_i$ satisfy (3), then by Proposition~\ref{lem-subs} with
$s\mapsto n, b_i\mapsto a_i$ for $i=1,\dots,n$ and $b_{n+1}\mapsto b$,
\[
x^{(a)}_m\Big|_{\substack{a_0=-\sum_{i=1}^na_i-b, \\ x_{i}=x_{n}q^{k_n-k_i},0\leq i\leq n}}
\]
is of the form
\[
-(q^{n_1}+\cdots+q^{n_b})x_n.
\]
Here $\{n_1,\dots,n_b\}$ is a set of integers.
It follows that
\begin{equation}\label{factor}
h_{v_0}\big(x^{(a)}_m\big)\Big|_{\substack{a_0=-\sum_{i=1}^na_i-b, \\ x_{i}=x_{n}q^{k_n-k_i},0\leq i\leq n}}
\end{equation}
is of the form
\[
h_{v_0}\big[{-x_n}(q^{n_1}+\cdots+q^{n_b})\big],
\]
which vanishes for $b<v_0$ by Proposition~\ref{prop-hr}.
Therefore, $Q(d\Mid 1,\dots,n;k_1,\dots,k_n)=0$ because it has \eqref{factor} as a factor.  Consequently, \eqref{qh5} holds.

To prove \eqref{qh4} for $1\leq s<n$ and the $k_i$ satisfy (3), we need Proposition~\ref{prop-rationalform}.
It shows that $Q(d\Mid u;k)$ is a rational function of the form \eqref{e-defF} with respect to $x_{u_s}$. Then we can apply Lemma~\ref{lem-almostprop} to eliminate the variable $x_{u_s}$ in $Q(d\Mid u;k)$.

If $1\leq s<n$ and $u_s=n$, then applying Lemma~\ref{lem-almostprop} yields
$\CT_{x_n}Q(d\Mid u;k)=0$, since there is no variable in $Q(d\Mid u;k)$ with index larger than $n$. Therefore, \eqref{qh4} holds for this case.

For $1\leq s<n$ and $u_s<n$, \eqref{qh4} holds if we can show that
\begin{equation}\label{qh6}
\CT_{x_{u_s}}Q(d\Mid u;k)=
\displaystyle \sum_{\substack{u_s<u_{s+1}\leq n,\\
1\leq k_{s+1}\leq d}}
Q(d\Mid u_1,\dots,u_s,u_{s+1};k_1,\dots,k_s,k_{s+1}).
\end{equation}
For any rational function $F$ of $x_{u_s}$ and integers $j$ and $z$, let
$T_{j,z} F$ be the result of replacing $x_{u_s}$ with
$x_{j}q^{z-k_s}$ in $F$. Since $Q(d\Mid u;k)$ is a rational function of the form \eqref{e-defF} with respect to $x_{u_s}$ by Proposition~\ref{prop-rationalform}, applying Lemma \ref{lem-almostprop} gives
\begin{equation}\label{e-TQ}
\CT_{x_{u_s}}Q(d\Mid u;k)=\sum_{\substack{u_s < u_{s+1}\le n \\ 1\le
k_{{s+1}}\le d}} T_{u_{s+1},k_{s+1}} \bigg(Q(d\Mid u;k)
\Big(1-\frac{x_{u_s}q^{k_s}}{x_{u_{s+1}}q^{k_{s+1}}}\Big)\bigg).
\end{equation}
To prove \eqref{qh6}, it suffices to show that
\[
Q(d\Mid u';k')
=T_{u_{s+1},k_{s+1}} \bigg(Q(d\Mid u;k)
\Big(1-\frac{x_{u_s}q^{k_s}}{x_{u_{s+1}}q^{k_{s+1}}}\Big)\bigg),
\]
where $u'=(u_1,\dots,
u_s, u_{s+1})$ and $k'=(k_1,\dots, k_s, k_{s+1})$.
The equality follows easily from the identity
\begin{equation}
\label{e-TE} T_{u_{s+1},k_{s+1}}\circ E_{u,k}= E_{u',k'}.
\end{equation}
To see that \eqref{e-TE} holds, we have
\[
(T_{u_{s+1},k_{s+1}}\circ E_{u,k})\, x_{u_i}
  =T_{u_{s+1},k_{s+1}}\, \Big( x_{u_s}q^{k_s-k_i}\Big)
  = x_{u_{s+1}}q^{k_{s+1}-k_i}= E_{u',k'}\,  x_{u_i},
\]
and if $j\notin \{u_0,u_1,\dots,u_s\}$ then $(T_{u_{s+1},k_{s+1}}\circ
E_{u,k})\, x_{j}=x_j= E_{u',k'}\, x_{j}$.
\end{proof}
We complete the proof of Lemma~\ref{lem-lead1} by proving the next proposition.
\begin{prop}\label{prop-rationalform}
Let $a_1,\dots,a_n$ be non-negative integers.
Let $s,d,u,k$ and $Q(d\Mid u;k)$ be defined as in \eqref{qh3}
such that $1\leq s<n$ and $d>a_{u_1}+\cdots+a_{u_s}$.
If the $k_i$ satisfy (3)
in the proof of Lemma~\ref{lem-lead1}, then
$Q(d\Mid u;k)$ is a rational function of the form \eqref{e-defF} with respect to $x_{u_s}$.
\end{prop}
\begin{proof}
Write $Q(d\Mid u;k)$ as $N/D$, in which $N$ (the ``numerator") is
\begin{multline*}
E_{u,k}\, \Bigg(h_{v^+}\bigg[\frac{1-q^{-d}}{1-q}x_0q^{-\chi(0<m)}
+\sum_{i=1}^n\frac{1-q^{a_i}}{1-q}x_iq^{-\chi(i<m)}\bigg] \\
\times x^{-v}\prod_{j=1}^n(qx_j/x_0)_{a_j}
\prod_{\substack{1\le i, j\le n\\ j\neq i}}
(q^{\chi(i>j)}x_i/x_j)_{a_i}\Bigg),
\end{multline*}
and
$D$ (the ``denominator") is
$$E_{u,k}\, \bigg(\prod_{j=1}^n (q^{-d}x_0/x_j)_{d}\biggm/\prod_{i=1}^s(1-q^{-k_i}
x_0/x_{u_i})\bigg).$$

Since $d>a_{u_1}+\cdots+a_{u_s}$, let $d=a_{u_1}+\cdots+a_{u_s}+b$ for a positive integer $b$.
Notice that
\begin{equation}\label{hv}
E_{u,k}\Bigg(h_{v^+}\bigg[\frac{1-q^{-d}}{1-q}x_0q^{-\chi(0<m)}
+\sum_{i=1}^n\frac{1-q^{a_i}}{1-q}x_iq^{-\chi(i<m)}\bigg]\Bigg)
\end{equation}
can be written as
\[
h_{v^+}(x_m^{(a)})\Big|_{\substack{a_0=-d=-\sum_{i=1}^sa_{u_i}-b, \\ x_{u_i}=x_{u_s}q^{k_s-k_i},0\leq i\leq s}}.
\]
It is of the form
\[
h_{v^+}\big[{-x_{u_s}}(q^{n_1}+\cdots+q^{n_b})+Y\big]
\]
by Proposition~\ref{lem-subs} with $b_i\mapsto a_{u_i}$ for $i=1,\dots,s$ and $b_{s+1}\mapsto b$ if the $k_i$ satisfy (3) in the proof of Lemma~\ref{lem-lead1}. Here $\{n_1,\dots,n_b\}$ is a set of integers and $Y=\sum_{i\notin U} (1-q^{a_i})x_iq^{-\chi(i<m)}/(1-q)$ is an alphabet independent of $x_{u_s}$, where $U:=\{u_0,u_1,\dots,u_s\}$ and $u_0=0$.
Thus, \eqref{hv} is a polynomial in $x_{u_s}$ of degree at most $\sum_{i=0}^{n}\min \{v_i,b\}$ by Proposition~\ref{prop-hr}.
It follows that the parts of $N$ contributing to the degree in
$x_{u_s}$,
\begin{multline*}
E_{u,k}\Bigg(h_{v^+}\bigg[\frac{1-q^{-d}}{1-q}x_0q^{-\chi(0<m)}
+\sum_{i=1}^n\frac{1-q^{a_i}}{1-q}x_iq^{-\chi(i<m)}\bigg]\\
\times\prod_{i\in U}x_i^{-v_i}\prod_{i=1}^s \prod_{j\notin U}
(q^{\chi(u_i>j)}x_{u_i}/x_j)_{a_{u_i}}\Bigg),
\end{multline*}
has degree at most
\[
\sum_{i=0}^{n}\min \{v_i,b\}-\sum_{i\in U}v_i+(n-s)(a_{u_1}+\cdots+a_{u_s}).
\]
The parts of $D$ contributing to the degree in $x_{u_s}$ are
\[
E_{u,k}\bigg(\prod_{j\notin U}(q^{-d}x_0/x_j)_{d}\bigg),
\]
which has degree $(n-s)d$.
Let $TD$ be the difference between the degrees in $x_{u_s}$ of $N$ and $D$.
\begin{align*}
TD:&=\sum_{i=0}^{n}\min \{v_i,b\}-\sum_{i\in U}v_i+(n-s)(a_{u_1}+\cdots+a_{u_s}-d)
\nonumber \\
&=\sum_{i=0}^{n}\min \{v_i,b\}-\sum_{i\in U}v_i-(n-s)b.
\end{align*}
Denote $\lambda=(\lambda_0,\dots,\lambda_n)=v^+$. Since $v_0$ is the unique largest part of $v$, $\lambda_0=v_0>\lambda_i$ for $i=1,\dots,n$. For $1\leq s<n$
\[
TD\leq \sum_{i=n-s}^{n}\min\{\lambda_i,b\}-\lambda_0-\sum_{i=n-s+1}^n \lambda_{i}
\leq \sum_{i=n-s}^{n}\lambda_i-\lambda_0-\sum_{i=n-s+1}^n \lambda_{i}
=\lambda_{n-s}-\lambda_0<0.
\]
Consequently, $Q(d\Mid u;k)$ is a rational function of the form \eqref{e-defF} with respect to $x_{u_s}$.
\end{proof}

Now we are ready to determine the roots of $D_{v,v^+}(a,m)$.
\begin{lem}\label{lem-vanishing}
Let $(a_1,\dots,a_n)$ and $v=(v_0,\dots,v_n)$ be compositions such that $v_0=\max\{v\}$ has multiplicity one in $v$.
For $-a_0\in \{0,1,\dots,a_1+\cdots+a_n+v_0-1\}$
and $m\in \{0,1,\dots,n+1\}$,
\begin{equation}
D_{v,v^+}(a,m)=0.
\end{equation}
\end{lem}
Note that $D_{v,v^+}(a,m)$ is a polynomial in $q^{a_0}$ of degree at most $a_1+\cdots+a_n+v_0$ for fixed non-negative integers $a_1,\dots,a_n$ by Proposition~\ref{cor-poly}.
Assuming the conditions of Lemma \ref{lem-vanishing},
for $D_{v,v^+}(a,m)$ viewed as a polynomial in $q^{a_0}$,
we find all its roots.
\begin{proof}
Since $\displaystyle\CT_x Q(-a_0)=D_{v,v^+}(a,m)$, we prove the lemma by showing that
\[
\CT_x Q(d)=0
\]
for $d\in \{0,\dots,a_1+\cdots+a_n+v_0-1\}$
under the assumptions of this lemma.
We have shown that $\displaystyle\CT_x Q(0)=0$.

We prove by induction on $n-s$ that
\[
\CT_x Q(d\Mid u;k) = 0\ \ \mbox{for}\ d\in \{1,\dots,a_1+\cdots+a_n+v_0-1\}.
\]
When $s=0$ this is what we need.
Note that taking constant term with respect to a variable that
does not appear has no effect.
We may assume that $s\le n$ and $0<u_1<\cdots<u_s\le n$, since
otherwise $Q(d\Mid u;k)$ is not defined.
If $s=n$ then $u_i=i$
for $i=1,2,\dots,n$. Thus,
\[
\CT_x Q(d\Mid u_1,\dots,u_n;k_1,\dots,k_n)
=\CT_x Q(d\Mid 1,\dots,n; k_{1},
\dots,k_{n})=0
\]
for $d\in \{1,\dots,a_1+\cdots+a_n+v_0-1\}$ by \eqref{e-Qn}.

Now suppose $0\le s<n$.  If part (i) of Lemma
\ref{lem-lead1} applies, then $Q(d\Mid u;k)=0$. Otherwise, part (ii) of Lemma \ref{lem-lead1} applies and \eqref{qh4} holds. Therefore, applying $\CT$ to both
sides of \eqref{qh4} gives
\[
\CT_x(d\Mid u;k)
=
\begin{cases}
\displaystyle\sum_{\substack{u_s < u_{s+1}\le n
\\ 1\le k_{{s+1}}\le d}}
\CT_x Q(d\Mid u_1,\dots, u_s, u_{s+1};k_1,\dots,k_s,k_{s+1}) &\text{for $u_s<n$,} \\
0 &\text{for $u_s=n$.}
\end{cases}
\]
By induction, every term in the above sum is zero, and so is the sum.
\end{proof}

Note that we can obtain a more general result by the similar argument as that about $D_{v,v^+}(a,m)$ in this subsection:
Let $\lambda=(\lambda_0,\lambda_1,\dots)$ be a partition and $v=(v_0,v_1,\dots,v_n)$ be a composition such that $\lambda\geq v^+$, $v_0=\max\{v\}$ and $\lambda_0>\max\{v_i\Mid i=1,\dots,n\}$. Then
\begin{equation}\label{e-generalorth}
D_{v,\lambda}(a,m)=0 \quad \text{for $-a_0\in \{0,\dots,a_1+\cdots+a_n+\lambda_0-1\}$.}
\end{equation}
Furthermore, for $D_{v,\lambda}(a,m)$ viewed as a polynomial in $q^{a_0}$,
if $\lambda_0>v_0$ then the number of roots exceeds its degree.
It follows that the polynomial $D_{v,\lambda}(a,m)$ is identically zero. Together with \eqref{relation-gamma}, we can conclude that $D_{v,\lambda}(a,m)\equiv 0$
for a partition $\lambda$ and a composition $v$ such that $\lambda\geq v^+$ and $\lambda_0>\max\{v\}$.
This contains the vanishing part of Conjecture~\ref{conj-Kadell} as a special case.
For $v\in \mathbb{Z}^{n+1}$, the argument about $D_{v,\lambda}(a,m)$ in this subsection is no longer valid in general.
For these cases, Cai obtained an orthogonality result,
see Proposition~\ref{prop-Cai} in the next subsection.

\subsection{The value of $D_{v,v^+}(a,m)$ at $a_0=1$}\label{sec7}

To determine $D_{v,v^+}(a,m)$, the last step is to obtain its one non-vanishing value at an additional point. In this subsection, we characterize the value of
$D_{v,v^+}(a,m)$ at $a_0=1$, and complete the proof of Theorem \ref{thm-2}.
We need a few results first.

By \eqref{relation-gamma}, it is easy to see that Theorem \ref{thm-Cai} by Cai is equivalent to the next result.
\begin{prop}\label{prop-Cai}
Let $v\in\mathbb{Z}^{n+1}$ and $\lambda$
a partition such that $|v|=|\lambda|$.
If $D_{v,\lambda}(a,m)$ is non-vanishing, then
$v^{+}\geq \lambda$.
\end{prop}
Note that we have a way to avoid using Proposition~\ref{prop-Cai} hinted by Cai's result in this paper, but the method is too complicated to present here.

By Proposition \ref{prop-Cai}, we show that the following constant terms vanish.
\begin{lem}\label{lem-hvanish}
Let $(a_1,\dots,a_n)$ and $(v_1,\dots,v_n)$ be compositions.
For $r$ an integer such that $r>\max\{v_i\Mid i=1,\dots,n\}+1$,
\begin{equation}\label{e-hvanish}
\CT_x \frac{h_r[\hat{x}_m^{(a)}]}{x_0^rx_1^{v_1}\cdots x_n^{v_n}}\prod_{i=1}^{n}h_{v_i}[x_0/q+\hat{x}_m^{(a)}](1-x_0/x_i)(qx_i/x_0)_{a_i}
\prod_{1\leq i<j\leq n} (x_i/x_j)_{a_i}(qx_j/x_i)_{a_j}=0,
\end{equation}
where
\[
\hat{x}_m^{(a)}:=\sum_{i=1}^n \frac{1-q^{a_i}}{1-q}\, x_i q^{-\chi(i<m)},
\]
and $m\in \{1,2,\dots,n+1\}$.
\end{lem}
\begin{proof}
By repeated use of \eqref{e-xy} with $X\mapsto x_0/q$ and $Y\mapsto \hat{x}_m^{(a)}$ we can expand $\prod_{i=1}^nh_{v_i}[x_0/q+\hat{x}_m^{(a)}]$ as
\[
\sum_{\substack{0\leq k_i\leq v_i\\
 1\leq i\leq n}}(x_0/q)^{\sum_{i=1}^nk_i}\prod_{i=1}^nh_{v_i-k_i}[\hat{x}_m^{(a)}].
\]
Together with the expansion
\begin{equation*}
\prod_{i=1}^{n}(1-x_0/x_i)=\sum_{s=0}^n \sum_{1\leq t_1<\cdots<t_s\leq n}(-1)^s \frac{x_{0}^s}{x_{t_1}\cdots x_{t_s}},
\end{equation*}
the constant term in \eqref{e-hvanish} becomes
\begin{align}\label{sum}
\sum (-1)^sq^{z}\CT_x \frac{h_{\mu}[\hat{x}_m^{(a)}]}{x_0^{\Delta}x_1^{v_1}\cdots x_n^{v_n}x_{t_1}\cdots x_{t_s}}\prod_{i=1}^{n}(qx_i/x_0)_{a_i}
\prod_{1\leq i<j\leq n} (x_i/x_j)_{a_i}(qx_j/x_i)_{a_j},
\end{align}
where
\[
z=-\sum_{i=1}^n k_i, \ \mu=(r,v_1-k_1,\dots,v_n-k_n)^+,\ \text{and}\
\Delta=r-s-\sum_{i=1}^n k_i,
\]
and the sum is over all integers $k_i$, $t_i$ and $s$ such that $0\leq k_i\leq v_i$ for $1\leq i\leq n$, $0\leq s\leq n$ and
$1\leq t_1<\cdots<t_s\leq n$.
Let $C$ be a constant term in \eqref{sum}.
We show that every $C$ equals 0 according to the sign of $\Delta$.

If $\Delta>0$, then $C=0$ by a degree consideration in $x_0$.

If $\Delta=0$, then
\begin{align*}
C&=\CT_x \frac{h_{\mu}[\hat{x}_m^{(a)}]}{x_1^{v_1}\cdots x_n^{v_n}x_{t_1}\cdots x_{t_s}}\prod_{i=1}^{n}(qx_i/x_0)_{a_i}
\prod_{1\leq i<j\leq n} (x_i/x_j)_{a_i}(qx_j/x_i)_{a_j}.
\end{align*}
By a degree consideration in $x_0$, it reduces to
\begin{align*}
C=\CT_x \frac{h_{\mu}[\hat{x}_m^{(a)}]}{x_1^{v_1}\cdots x_n^{v_n}x_{t_1}\cdots x_{t_s}}
\prod_{1\leq i<j\leq n} (x_i/x_j)_{a_i}(qx_j/x_i)_{a_j}.
\end{align*}
Let $w=(w_1,\dots,w_n)$ be the vector such that $x_1^{w_1}\cdots x_n^{w_n}={x_1^{v_1}\cdots x_n^{v_n}x_{t_1}\cdots x_{t_s}}$.
Then $C$ can be written as
\[
D_{w,\mu}(a^{(0)},m-1).
\]
Since the largest part of $\mu$ is $r$, which exceeds
the largest part of $w$ (that is at most
$\max\{v_i\Mid i=1,\dots,n\}+1$),
$w^+\geq \mu$ can not hold. Thus, we can conclude $C=0$ by Proposition \ref{prop-Cai}.

For $\Delta<0$,
let $w=(w_0,w_1,\dots,w_n)$ be the vector such that $x^w={x_0^{\Delta}x_1^{v_1}\cdots x_n^{v_n}x_{t_1}\cdots x_{t_s}}$.
Then
\[
C=D_{w,\mu}(a,m)\Mid_{a_0=0}.
\]
Because of the same reason as that for the $\Delta=0$ case, $w^+\geq \mu$ can not hold.
Thus, $D_{w,\mu}(a,m)=0$ by Proposition \ref{prop-Cai}.
It follows that
\[
C=D_{w,\mu}(a,m)\Mid_{a_0=0}\,=0.
\]

In summary, every summand in \eqref{sum} equals 0 and so is the sum.
\end{proof}

Using Lemma \ref{lem-hvanish} and the generating function of complete symmetric functions \eqref{e-gfcomplete},
we can characterize the value of $D_{v,v^+}(a,m)$ at $a_0=1$ for $v$ a composition such that $v_0=\max\{v\}$ has multiplicity one in $v$.
\begin{lem}\label{lem-iterate}
Let $v=(v_0,\dots,v_n)$ be a composition such that $v_0$ is its unique largest part.
For $m\in \{1,2,\dots,n+1\}$
\begin{equation}
D_{v,v^+}(a,m)|_{a_0=1}=
q^{-v_0+\sum_{i=1}^{m-1} a_i}D_{v^{(0)},(v^{(0)})^+}(a^{(0)},m-1).
\end{equation}
\end{lem}
\begin{proof}
By the definition of $D_{v,v^+}(a,m)$, for $m\in \{1,2,\dots,n+1\}$
\begin{align*}
D_{v,v^+}(a,m)|_{a_0=1}
=\CT_x  \frac{h_{v^+}\big[x_0/q+\hat{x}_m^{(a)}\big]}{x^{v}}
\prod_{i=1}^{n}(1-x_0/x_i)(qx_i/x_0)_{a_i}D_{n}(a^{(0)}),
\end{align*}
where $\hat{x}_m^{(a)}:=\sum_{i=1}^{n}x_i(1-q^{a_i})q^{-\chi(i<m)}/(1-q)$
and $D_{n}(a^{(0)}):=\prod_{1\leq i<j\leq n} (x_i/x_j)_{a_i}(qx_j/x_i)_{a_j}$.
By \eqref{e-xy} with $X\mapsto x_0/q$ and $Y\mapsto \hat{x}_m^{(a)}$, expand $h_{v_0}[x_0/q+\hat{x}_m^{(a)}]$ as
\begin{equation*}
\sum_{r=0}^{v_0}(x_0/q)^{v_0-r}h_{r}[\hat{x}_m^{(a)}].
\end{equation*}
Then
\begin{align*}
D_{v,v^+}(a,m)|_{a_0=1}
=q^{-v_0}\sum_{r=0}^{v_0}\CT_x
\frac{h_{r}[\hat{x}_m^{(a)}]h_{(v^{(0)})^+}[x_0/q+\hat{x}_m^{(a)}]}{(x_0/q)^rx_1^{v_1}\cdots x_n^{v_n}} \prod_{i=1}^{n}(1-x_0/x_i)(qx_i/x_0)_{a_i}D_{n}(a^{(0)}).
\end{align*}
If $r>v_0$, then $r>\max\{v_i\Mid i=1,\dots,n\}+1$
because $v_0$ is the unique largest part of $v$.
Hence, by Lemma \ref{lem-hvanish}
the constant term in the above sum equals 0 if $r>v_0$.
It follows that $D_{v,v^+}(a,m)|_{a_0=1}$ can be written as
\begin{align*}
q^{-v_0}\CT_x \sum_{r=0}^{\infty}
\frac{h_{r}[\hat{x}_m^{(a)}]h_{(v^{(0)})^+}[x_0/q+\hat{x}_m^{(a)}]}
{(x_0/q)^rx_1^{v_1}\cdots x_n^{v_n}}
\prod_{i=1}^{n}(1-x_0/x_i)(qx_i/x_0)_{a_i}D_{n}(a^{(0)}).
\end{align*}
By the generating function of complete symmetric functions \eqref{e-gfcomplete} this becomes
\begin{align*}
q^{-v_0}\CT_x \frac{h_{(v^{(0)})^+}[x_0/q+\hat{x}_m^{(a)}]\prod_{i=1}^{n}(1-x_0/x_i)(qx_i/x_0)_{a_i}
D_{n}(a^{(0)})}{\prod_{i=1}^{m-1}(x_i/x_0)_{a_i}
\prod_{i=m}^n(qx_i/x_0)_{a_i}x_1^{v_1}\cdots x_n^{v_n}}.
\end{align*}
Cancelling the same factors yields
\begin{multline*}
q^{-v_0}\CT_x \frac{h_{(v^{(0)})^+}[x_0/q+\hat{x}_m^{(a)}]}{x_1^{v_1}\cdots x_n^{v_n}}\prod_{i=1}^{m-1}\frac{(1-x_0/x_i)(1-q^{a_i}x_i/x_0)}{1-x_i/x_0}
\prod_{i=m}^n(1-x_0/x_i)D_{n}(a^{(0)}) \\
=q^{-v_0}\CT_x \frac{h_{(v^{(0)})^+}[x_0/q+\hat{x}_m^{(a)}]}{x_1^{v_1}\cdots x_n^{v_n}}\prod_{i=1}^{m-1}(q^{a_i}-x_0/x_i)
\prod_{i=m}^n(1-x_0/x_i)D_{n}(a^{(0)}).
\end{multline*}
By a degree consideration in $x_0$, it further reduces to
\[
q^{-v_0+\sum_{i=1}^{m-1}a_i}\CT_x \frac{h_{(v^{(0)})^+}[\hat{x}_m^{(a)}]}
{x_1^{v_1}\cdots x_n^{v_n}}D_{n}(a^{(0)}),
\]
which can be written as
\[
q^{-v_0+\sum_{i=1}^{m-1}a_i}D_{v^{(0)},(v^{(0)})^+}(a^{(0)},m-1).
\qedhere
\]
\end{proof}
Now we obtain all the ingredients for characterizing  $D_{v,v^+}(a,m)$ if $v$ is a composition such that $v_0=\max\{v\}$ has multiplicity one in $v$,
and $m\in \{1,2,\dots,n+1\}$.
By Proposition \ref{cor-poly}, $D_{v,v^+}(a,m)$ is a polynomial in $q^{a_0}$ of degree at most
$a_1+\cdots+a_n+v_0$ for fixed non-negative integers $a_1,\dots,a_n$.
By Lemma \ref{lem-iterate}
\[
D_{v,v^+}(a,m)|_{a_0=1}
=q^{-v_0+\sum_{i=1}^{m-1} a_i} D_{v^{(0)},(v^{(0)})^+}(a^{(0)},m-1).
\]
By Lemma \ref{lem-vanishing}
\[
D_{v,v^+}(a,m)=0
\]
for $-a_0\in \{0,1,\dots,a_1+\cdots+a_n+v_0-1\}$.
Hence, the above properties uniquely determine $D_{v,v^+}(a,m)$ as
\begin{equation}
D_{v,v^+}(a,m)=q^{-v_0+\sum_{i=1}^{m-1} a_i}\qbinom{v_0+|a|-1}{a_0-1}
D_{v^{(0)},(v^{(0)})^+}(a^{(0)},m-1).
\end{equation}
This completes the proof of Theorem \ref{thm-2}.

\end{document}